\documentclass[11pt]{amsart}

\usepackage[utf8]{inputenc}
\usepackage[usenames,dvipsnames]{xcolor}
\usepackage{url}
\usepackage{theoremref}
\usepackage{pinlabel}

\allowdisplaybreaks

\newcommand{\bbZ}{\mathbb{Z}}

\newcommand{\bbR}{\mathbb{R}}
\newcommand{\bbC}{\mathbb{C}}

\newcommand{\cross}{\times}

\newcommand{\xrightarrowt}[2]{\xrightarrow[\text{#2}]{\text{#1}}}

\theoremstyle{plain}
\newtheorem{theorem}{Theorem}[section]
\newtheorem{lemma}[theorem]{Lemma}
\newtheorem*{lemma*}{Lemma}

\newtheorem{corollary}[theorem]{Corollary}

\newtheorem{conjecture}[theorem]{Conjecture}
\newtheorem{question}[theorem]{Question}

\theoremstyle{definition}
\newtheorem{definition}[theorem]{Definition}

\theoremstyle{remark}

\newtheorem{example}[theorem]{Example}

\numberwithin{equation}{section}

\newcommand{\allbf}[1]{\textbf{\boldmath{#1}}}

\renewcommand{\r}{r}
\newcommand{\tb}{\operatorname{tb}}
\newcommand{\tw}{\operatorname{tw}}
\newcommand{\Wh}{W}

\newcommand{\sat}[3]{\ensuremath{S_{#1}(#2,#3)}}
\newcommand{\lsat}[2]{\ensuremath{\Sigma(#1,#2)}}

\title{Legendrian Satellites and Decomposable Concordances}

\author[Y. Liu]{Yanhan Liu}\address{Haverford College, Haverford, PA 19041} \email{yliu9@haverford.edu}

\author[J. Sabloff]{Joshua M. Sabloff} \address{Haverford College,
Haverford, PA 19041} \email{jsabloff@haverford.edu} \thanks{YL, JMS, and SZ were
partially supported by NSF grant DMS-1406093 in the preparation of this paper. MY was partially supported by a grant from Haverford College.}

\author[M. Yacavone]{Matthew Yacavone} \address{Haverford College, Haverford, PA 19041} \email{myacavone@haverford.edu}

\author[S. Zhou]{Sipeng Zhou} \address{Haverford College, Haverford, PA 19041} \email{szhou1@haverford.edu}

\date{\today}

\begin{document}

\begin{abstract}
    We investigate the ramifications of the Legendrian satellite construction on the relation of Lagrangian cobordism between Legendrian knots.  Under a simple hypothesis, we construct a Lagrangian concordance between two Legendrian satellites by stacking up a sequence of elementary cobordisms. This construction  narrows the search for ``non-decomposable'' Lagrangian cobordisms and yields new families of decomposable Lagrangian slice knots. Finally, we show that the maximum Thurston-Bennequin number of a smoothly slice knot provides an obstruction to any Legendrian satellite of that knot being Lagrangian slice.
\end{abstract}

\maketitle

\section{Introduction}\label{introduction}

Lagrangian cobordisms between Legendrian submanifolds --- especially exact cobordisms that lie in symplectizations --- have seen increasing attention in recent research.  Lagrangian cobordisms  are of interest not only as central objects in the relative Symplectic Field Theory \cite{c-dr-g-g, c-dr-g-g-cobordism, rizell:surgery, rizell:lifting, ekholm:rsft, ekholm:rsft-survey, ekholm:lagr-cob, ehk:leg-knot-lagr-cob, golovko:tb} and generating family \cite{bst:construct,josh-lisa:obstr} frameworks, but also because of their connections to the topology of smooth knots \cite{bty,polyfillability,chantraine, positivity}.  Subtle questions  have arisen about the structure of the relation that Lagrangian cobordism induces on the set of Legendrian submanifolds. The goal of this paper is to use the Legendrian satellite construction to illuminate both questions about the Lagrangian cobordism relation and questions of when a smoothly slice knot has a Legendrian representative that is Lagrangian slice.

\subsection{The Lagrangian Cobordism Relation}

We begin by discussing the Lagrangian cobordism relation.  In contrast to the smooth setting, where cobordism --- and even concordance --- between knots induces an equivalence relation, Lagrangian cobordism between Legendrians is not a symmetric relation  \cite{bs:monopole,chantraine:non-symm, cns:obstr-concordance}, though it is not clear whether Lagrangian concordance induces a partial order.  Further, to achieve transitivity, one must impose the condition that, outside a compact set,  Lagrangian cobordisms are equal to cylinders over Legendrians (see Section~\ref{cobordismsBkg} for the precise definition). This cylindrical-at-infinity condition turns out to be rather delicate: not all Lagrangian surfaces in $\bbC^2$ that meet the standard contact $3$-sphere in a Legendrian knot, for example, can be deformed to be cylindrical-at-infinity \cite{chantraine:collar}, though the known examples are not exact. Pushing this line of thought further, consider that in the smooth category, a generic cobordism may always be decomposed into a sequence of elementary cobordisms arising from smooth isotopy and from attaching a single handle.  It is unclear whether cylindrical-at-infinity Lagrangian cobordisms share this feature.  On one hand, there are constructions of elementary Lagrangian cobordisms \cite{bst:construct,rizell:surgery, ehk:leg-knot-lagr-cob} (see Theorem~\ref{thm:construct}, below);  we say that a Lagrangian cobordism constructed by stacking  elementary cobordisms is \allbf{decomposable}. On the other, Chantraine's example noted above hints that it might not be the case that all Lagrangian cobordisms are decomposable.

One possible source of indecomposable Lagrangian cobordisms arises from the Legendrian satellite construction.  We briefly set notation:  given a Legendrian knot $\Lambda$ in the standard contact $\bbR^3$ and a Legendrian link $\Pi$ in the solid torus $J^1S^1$, we denote the Legendrian satellite of $\Lambda$ by $\Pi$ by $\lsat{\Lambda}{\Pi}$. Given a Lagrangian concordance, one can find, via a geometric construction in \cite[Theorem 2.4]{cns:obstr-concordance}, a Lagrangian concordance between Legendrian satellites of the knots at the ends. It is not at all clear that the results of this construction are decomposable.  In fact,  Cornwell, Ng, and Sivek \cite[Conjecture 3.3]{cns:obstr-concordance} put forward a candidate for an indecomposable Lagrangian concordance between Whitehead doubles of concordant Legendrians.  One might ask how careful they had to be in their construction; that is:

\begin{question} \label{qu:decomposable}
    Under what conditions does there exist a decomposable Lagrangian concordance between Legendrian satellites?
\end{question}

The main result of the paper supplies one such condition.  To state the result, we consider the Legendrian solid torus link $\tw_m$, an example of which is depicted in Figure~\ref{fig:wh-tw}(b).  The sum of two Legendrian solid torus links comes from simply concatenating front diagrams; see the end of Section~\ref{solid-torusBkg}.

\begin{theorem}\label{mainThrm}
    If $\Pi$ is Legendrian solid torus link  with $m$ strands and the Legendrian knots $\Lambda_-$ and $\Lambda_+$ are decomposably concordant, then $\lsat{\Lambda_-}{\Pi +  \tw_m}$ and $\lsat{\Lambda_-}{\Pi + \tw_m}$ are also decomposably concordant.
\end{theorem}

With this result in hand, we sharpen Question~\ref{qu:decomposable}, above:

\begin{question}\label{mainQuestion}
If the Legendrian knots $\Lambda_-$ and $\Lambda_+$ are decomposably concordant and $\lsat{\Lambda_-}{\Pi}$ and $\lsat{\Lambda_+}{\Pi}$ are decomposably concordant, must the Legendrian solid torus link $\Pi$ be isotopic to $\Pi' + \tw_m$ for some $\Pi'$?
\end{question}

A proof of the affirmative would imply Conjecture 3.3 in \cite{cns:obstr-concordance}.

\subsection{Lagrangian Slice Knots}

Shifting our focus slightly to interactions between smooth and Lagrangian cobordism, recall that satellite constructions have a long history as tools for investigating slice knots and smooth concordance; see, for example, the  recent papers \cite{cdr:inj-sat, chl:fractal-concordance,hk:wh}. We use Legendrian satellites to investigate the following question \cite{cns:obstr-concordance}:

\begin{question}
When does a smoothly slice knot have a Legendrian representative with Lagrangian slice disk?
\end{question}

It is straightforward to produce new Lagrangian slice knots from known ones using \cite[Theorem 2.4]{cns:obstr-concordance}.  To select an appropriate pattern $\Pi \subset J^1S^1$, let $\Upsilon$ denote the maximal Legendrian unknot, and suppose that the solid torus knot $\Pi$ has the property that $\lsat{\Upsilon}{\Pi}$ is Lagrangian slice. For example, if $\Wh$ is the Legendrian Whitehead double pattern (see Figure~\ref{fig:wh-tw}(a)), then $\Pi = \Wh + \tw_2$ has the property that $\lsat{\Upsilon}{\Pi}$ is isotopic to $\Upsilon$ and hence is Lagrangian slice. If there is a concordance from $\Upsilon$ to a Legendrian knot $\Lambda$, then the concordance constructed from \cite[Theorem 2.4]{cns:obstr-concordance} stacked on top of the Lagrangian slice disk for $\lsat{\Upsilon}{\Pi}$ shows that $\lsat{\Lambda}{\Pi}$ is Lagrangian slice.  Concretely, this yields many examples of Lagrangian slice knots, such as those exhibited in \cite[\S4]{cns:obstr-concordance} or iterations of the Whitehead double construction applied to any of these knots. Note that the condition on $\Lambda$ is somewhat stronger than $\Lambda$ merely being Lagrangian slice: if $\Lambda$ is Lagrangian slice, there might not be a way to decompose the Lagrangian slice disk into a concordance to $\Upsilon$ stacked on top of a $0$-handle. 

We can use Theorem~\ref{mainThrm} to produce new decomposable Lagrangian slice knots:

\begin{corollary}\label{sliceCorr}
    If $\Lambda$ is a decomposable Lagrangian slice knot and $\Pi$ is a Legendrian solid torus knot such that $\lsat{\Upsilon}{\Pi+\tw_m}$ is decomposable Lagrangian slice, then $\lsat{\Lambda}{\Pi+\tw_m}$ is decomposable Lagrangian slice.
\end{corollary}

One can ask if Whitehead doubling a smoothly slice knot that does \emph{not} have a Lagrangian slice representative can yield a new slice knot that does, indeed, have a Lagrangian slice representative.  The second main theorem of this paper shows that this is not possible if the maximal Thurston-Bennequin number of the original slice knot is not $-1$.

\begin{theorem} \label{thm:no-new-slice}
    Suppose $P$ is a smooth solid torus knot for which the untwisted satellite of $P$ around the unknot $U$ is slice.  If $K$ is a slice knot with maximal Thurston-Bennequin number strictly less than $-1$, then the untwisted satellite of $P$ around the  $K$ has no Lagrangian slice representatives.
\end{theorem}

Conjecture 4.1 of \cite{cns:obstr-concordance} asserts that a smoothly slice knot has a Lagrangian slice representative if and only if its maximal Thurston-Bennequin number is $-1$.  If the conjecture holds, then Theorem~\ref{thm:no-new-slice} can be strengthened to say that one cannot produce new Lagrangian slice knots by taking satellites of knots that are not Lagrangian slice.  Theorem~\ref{thm:no-new-slice} may also be thought of as providing evidence for Conjecture 4.1 of \cite{cns:obstr-concordance}.

Pushing further, one might reasonably adapt the old question of whether a knot is smoothly slice if and only if its untwisted Whitehead double is (see Question 1.38 in Kirby's problem list \cite{kirby:list}) to the Legendrian setting:

\begin{conjecture}
    A Legendrian knot $\Lambda$ is Lagrangian slice if and only if its untwisted Whitehead double $\lsat{\Lambda}{\Wh+\tw_2}$ is Lagrangian slice.
\end{conjecture}

The rest of the paper is organized as follows:  we review necessary background on Legendrian solid torus knots, the Legendrian satellite construction, and Lagrangian cobordisms between Legendrians in Section~\ref{background}. We then proceed to prove Theorem~\ref{mainThrm} in Section~\ref{mainProof} and Theorem~\ref{thm:no-new-slice} in Section~\ref{sliceCorr}.

\section{Background}\label{background}

In this section, we describe the two central constructions in this paper:  Legendrian satellites and decomposable Lagrangian cobordisms.  We assume familiarity with basic notions of Legendrian knot theory in $\bbR^3$ (as in \cite{etnyre:knot-intro}) and of Lagrangian submanifolds (as in \cite{audin-lalonde-polterovich}), though we quickly review Legendrian links in the solid torus $J^1S^1$.

\subsection{Legendrian Solid Torus Links} \label{solid-torusBkg}

The canonical solid torus in contact topology is the $1$-jet space $J^1S^1 = S^1 \times \bbR^2$, which has coordinates $(\theta, y, z)$ and contact form $dz-y\,d\theta$.  A Legendrian link $\Pi \subset J^1S^1$ may be represented by its front diagram, i.e. its image under the projection $\pi_f(\theta, y, z) = (\theta,z)$.  It is convenient to draw front diagrams in $[0,1] \times \bbR$, with $\{0\} \times \bbR$ identified with $\{1\} \times \bbR$ as in Figure~\ref{fig:torus-front}; the Legendrian solid torus links $\Wh$ and $\tw_m$ depicted in Figure~\ref{fig:wh-tw} will be of particular interest in this paper.

\begin{figure}
    \centering
    \includegraphics{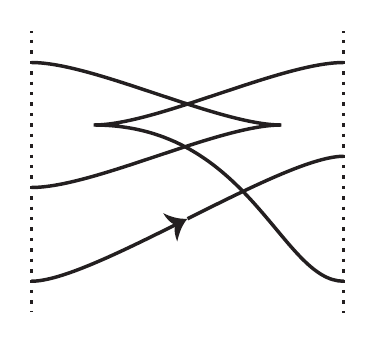}
    \caption{A front diagram for a Legendrian link in the solid torus $J^1S^1$.}
    \label{fig:torus-front}
\end{figure}

\begin{figure}
\labellist
    \small
    \pinlabel (a) [l] at 35 0
    \pinlabel (b) [l] at 180 0
    \endlabellist

    \centering
    \includegraphics{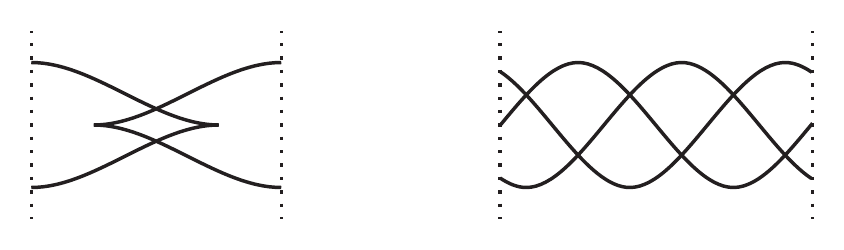}
    \caption{The Whitehead solid torus knot $\Wh$ (a) and the full $m$-twist link $\tw_m$ with $m=3$ (b).} 
    \label{fig:wh-tw}
\end{figure}

There are three invariant quantities we may associate to an oriented Legendrian link $\Pi \subset J^1S^1$:
\begin{enumerate}
\item The \allbf{Thurston-Bennequin number} $\tb(\Pi)$ is computed from the front diagram as the writhe minus half the number of cusps.
\item The \allbf{rotation number} $\r(\Pi)$ is computed from the front diagram as half the difference between the number of cusps at which $\pi_f(\Pi)$ is oriented downwards and the number of cusps at which $\pi_f(\Pi)$ is oriented upwards.
\item The \allbf{winding number} or \allbf{degree} $\delta(\Pi)$ is the signed intersection number between the front diagram and $\{0\} \times \bbR$; this is, of course, simply the homology class of $\Pi$ in $H_1(J^1S^1) \simeq \bbZ$. 
\end{enumerate}
For a fixed front diagram of $\Pi$, we can also define the \allbf{strand number} $s(\pi_f(\Pi))$ to be the geometric intersection number of the front diagram with $\{0\} \times \bbR$. 

\begin{example}
The front diagram in Figure~\ref{fig:torus-front} represents a Legendrian knot $\Pi$ with $\tb(\Pi) = 2$, $\r(\Pi) = 0$, $\delta(\Pi) = 1$, and $s(\pi_f(\Pi)) = 3$. It is also easy to see from the diagram that $\tb(\Wh) = 1$.
\end{example}

\begin{example}
Consider a full twist $\tw_m$ with a given orientation. A somewhat involved but elementary diagrammatic argument allows us to compute the Thurston-Bennequin number of a full twist:
\begin{equation} \label{eq:tb-tw}
 \tb(\tw_m) = \delta(\tw_m)^2-m.  
\end{equation}
Note that the value of $\delta(\tw_m)^2$ is dependent on the specific orientations of the strands in $\tw_m$.
\end{example}

The operation of adding a full twist to an existing Legendrian solid torus knot plays a key role in the statement of Theorem~\ref{mainThrm}.  More generally, if two Legendrian links $\Pi_1$ and $\Pi_2$ have front diagrams with the same strand number, then we can form a new Legendrian link $\Pi_1 + \Pi_2$ by concatenating the front diagrams, with $\pi_f(\Pi_1)$ to the left of $\pi_f(\Pi_2)$.  We will be particularly interested in sums of the form $\Pi + \tw_m$.

\subsection{Satellites}\label{satellitesBkg}

We begin by setting notation for the satellite of a knot in the smooth category. Let $K$ be a smooth knot in $\bbR^3$, let $P$ be a smooth link in the solid torus $S^1 \times D^2$, and let $t \in \bbZ$.  We define the \allbf{$t$-framed satellite} of $K$ by $P$, denoted $\sat{t}{K}{P}$, to be the link that results from removing a normal neighborhood of $K$ and gluing in the solid torus containing $P$ with framing $t$ with respect to the Seifert framing of $K$.  This operation is easily generalized when $K$ is a link with $n$ ordered components, $\mathbf{P} = (P_1, \ldots, P_n)$ is an $n$-tuple of solid torus links, and $\mathbf{t} = (t_1, \ldots, t_n)$ is an $n$-tuple of integers corresponding to framings with respect to Seifert surfaces of individual components.

To define Legendrian satellite of a Legendrian knot $\Lambda \subset \bbR^3$  by a Legendrian link $\Pi \subset J^1S^1$, we follow \cite{ng:satellite, lenny-lisa}. It is a fundamental fact that $\Lambda$ has a canonical neighborhood that is contactomorphic to $J^1S^1$.  We define the \allbf{Legendrian satellite} of $\Lambda$ by $\Pi$, denoted $\lsat{\Lambda}{\Pi}$, to be the link that results from removing a standard neighborhood of $\Lambda$ and gluing in the solid torus containing $\Pi$.  Again, there is an obvious generalization to satellites of Legendrian links. Note that, unlike in the smooth case, the Legendrian satellite has a built-in framing; in particular, as smooth knots, we have $\lsat{\Lambda}{\Pi} = \sat{\tb(\Lambda)}{\Lambda}{\Pi}$.  

The following lemma, which essentially states that the satellite operation is well-defined, is straightforward from the definition of a Legendrian satellite.

\begin{lemma}[\cite{lenny-lisa}]\label{satLiftIso}
Given $n$-component Legendrian links $\Lambda,\Lambda' \subset \bbR^3$ and $n$-tuples of Legendrian links $\mathbf{\Pi}, \mathbf{\Pi}' \subset J^1S^1$, if $\Lambda$ is Legendrian isotopic to $\Lambda'$ and $\mathbf{\Pi}$ is Legendrian isotopic to $\mathbf{\Pi}'$, then $\lsat{\Lambda}{\mathbf{\Pi}}$ is Legendrian isotopic to $\lsat{\Lambda'}{\mathbf{\Pi}'}$. 
\end{lemma}

At a more practical level, Figure~\ref{fig:front-sat} depicts a straightforward procedure for drawing a front diagram of $\lsat{\Lambda}{\Pi}$ from front diagrams of $\Lambda$ and of $\Pi$.  Suppose that $m = s(\pi_f(\Pi))$. Start by taking the $m$-copy of a front diagram of $\Lambda$, i.e. $m$ copies of $\Lambda$ that differ by a small vertical shift. Cut open the front diagram of $\pi_f(\Pi)$ along $\{0\} \times \bbR$ to create a diagram in $[0,1] \times \bbR$. Replace the $m$-stranded tangle created by the $m$-copy of a small interval in the front of $\Lambda$ (away from cusps and crossings) with the cut open front diagram $\pi_f(\Pi)$.  The result is a front diagram for $\lsat{\Lambda}{\Pi}$.

\begin{figure}
\labellist
    \small
    \pinlabel $\Lambda$ [l] at 140 120
    \pinlabel $\Pi$ [l] at 140 10
    \pinlabel $\lsat{\Lambda}{\Pi}$ [l] at 380 10
\endlabellist

    \centering
    \includegraphics[width=4.5in]{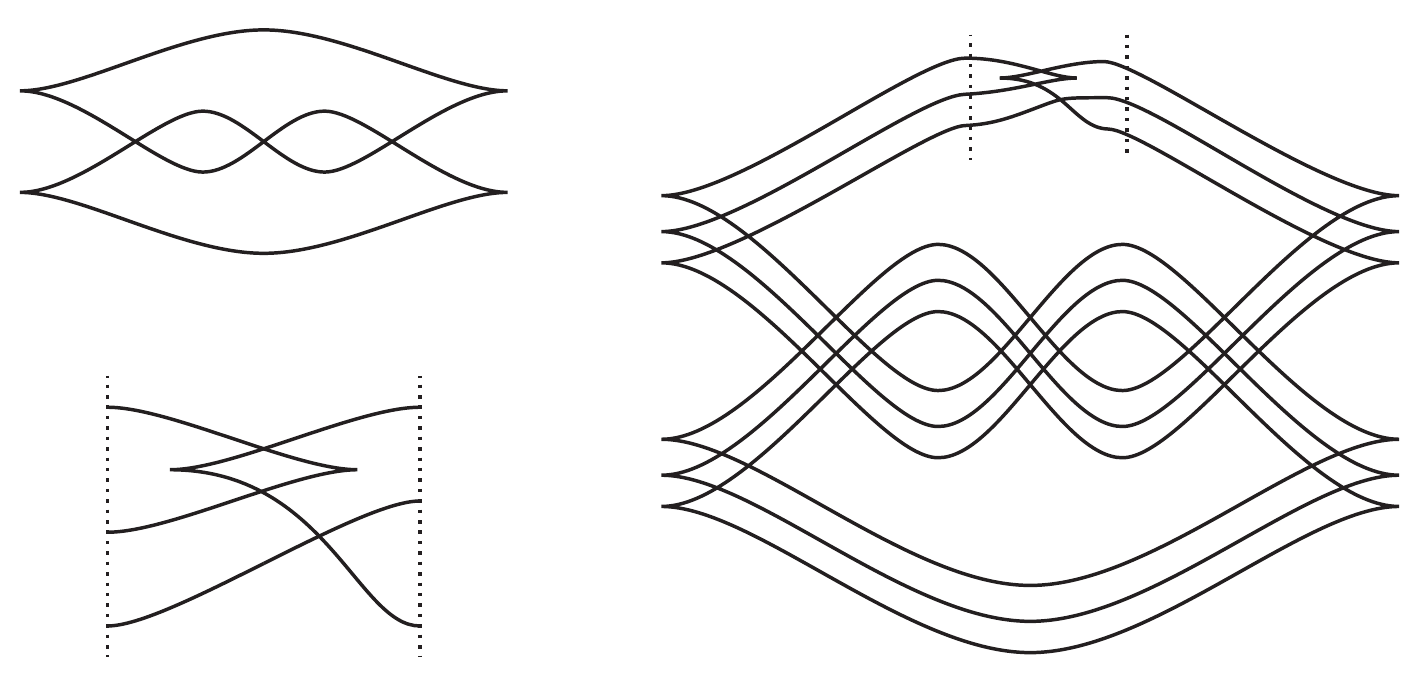}
    \caption{A procedure for creating a front diagram for $\lsat{\Lambda}{\Pi}$.} 
    \label{fig:front-sat}
\end{figure}

It is straightforward to compute the classical invariants of a Legendrian satellite from those of its consitutents, as noted in \cite[Remark 2.4]{ng:satellite}:

\begin{lemma} \label{lem:sat-tb}
    For a Legendrian knot $\Lambda$ and a Legendrian solid torus link $\Pi$, the classical invariants of the satellite $\lsat{\Lambda}{\Pi}$ may be computed as follows:
    \begin{align*}
        \tb(\lsat{\Lambda}{\Pi}) &= \delta(\Pi)^2 \tb(\Lambda) + \tb(\Pi),\\
        \r(\lsat{\Lambda}{\Pi}) &= \delta(\Pi) \r(\Lambda) + \r(\Pi).
    \end{align*}
\end{lemma}

\subsection{Decomposable Cobordisms and Concordances}\label{cobordismsBkg}

The second fundamental notion in this paper is that of a Lagrangian cobordism between Legendrians.  

\begin{definition} \label{defn:lagr-cob}
    Given two Legendrian links $\Lambda_-$ and $\Lambda_+$ in the standard contact $\bbR^3$, an \allbf{(exact, orientable, cylindrical-at-infinity) Lagrangian cobordism} $L$ from $\Lambda_-$ to $\Lambda_+$, denoted $\Lambda_- \prec_L \Lambda_+$, is an exact, orientable Lagrangian submanifold of the symplectization $(\bbR \times \bbR^3, d(e^t \alpha_0))$ so that there is a pair of real numbers $T_\pm$ satisfying
    \begin{enumerate}
        \item $L \cap (-\infty,T_-] \times \bbR^3 = (-\infty,T_-] \times \Lambda_-$,
        \item $L \cap [T_+,\infty) \times \bbR^3 = [T_+,\infty) \times \Lambda_+$, and
        \item The primitive of $e^t\alpha_0$ along $L$ is constant for $t < T_-$ and for $t> T_+$.  See \cite{chantraine:disconnected-ends} for an explanation of the necessity of this condition.
    \end{enumerate}
\end{definition}

We will drop the descriptors ``exact, orientable, and cylindrical-at-infinity'' from this point on, as they shall be understood to hold. Two special cases will play important roles in this paper. A \allbf{Lagrangian concordance} is a Lagrangian cobordism that is diffeomorphic to a cylinder. A \allbf{Lagrangian filling} is a Lagrangian cobordism with an empty negative end.  There is a close relationship between Lagrangian fillings and the slice genus of a smooth knot.  On one hand, the slice-Bennequin inequality \cite{rudolph:qp-obstruction} says that 
\begin{equation} \label{eq:slice-bennequin}
    \overline{\tb}(K) \leq 2g_4(K) -1.
\end{equation}
On the other, if a smooth knot $K$ has a Legendrian representative $\Lambda$ with a Lagrangian filling $L$, then Chantraine proved in \cite{chantraine} that the slice-Bennequin inequality is sharp and, in fact, 
\begin{equation} \label{eq:filling-g4}
g(L) = g_4(K).
\end{equation}

A Lagrangian cobordism is \allbf{decomposable} if it is the result of stacking the elementary cobordisms defined by the following theorem.

\begin{theorem}[\cite{bst:construct, rizell:surgery, ehk:leg-knot-lagr-cob}] \label{thm:construct}
    If two Legendrian links $\Lambda_-$ and $\Lambda_+$ in $\bbR^3$ are related by any of the following moves, which we shall call \allbf{elementary cobordism moves}, then there exists an exact, orientable Lagrangian cobordism $\Lambda_- \prec_L \Lambda_+$.
    \begin{description}
        \item[Isotopy] $\Lambda_-$ is Legendrian isotopic to $\Lambda_+$,
        \item[$0$-Handle] The front diagrams for $\Lambda_-$ and $\Lambda_+$ are identical except for the addition of a disjoint Legendrian unknot in $\Lambda_+$ as in Figure~\ref{fig:construct} (left).
        \item[$1$-Handle] The front diagrams for $\Lambda_-$ and $\Lambda_+$ are related as in Figure~\ref{fig:construct} (right).
    \end{description}
\end{theorem}

\begin{figure}
\labellist
	\large
	\pinlabel $\emptyset$ [b] at 32 25
\endlabellist

    \centering
    \includegraphics[height=2in]{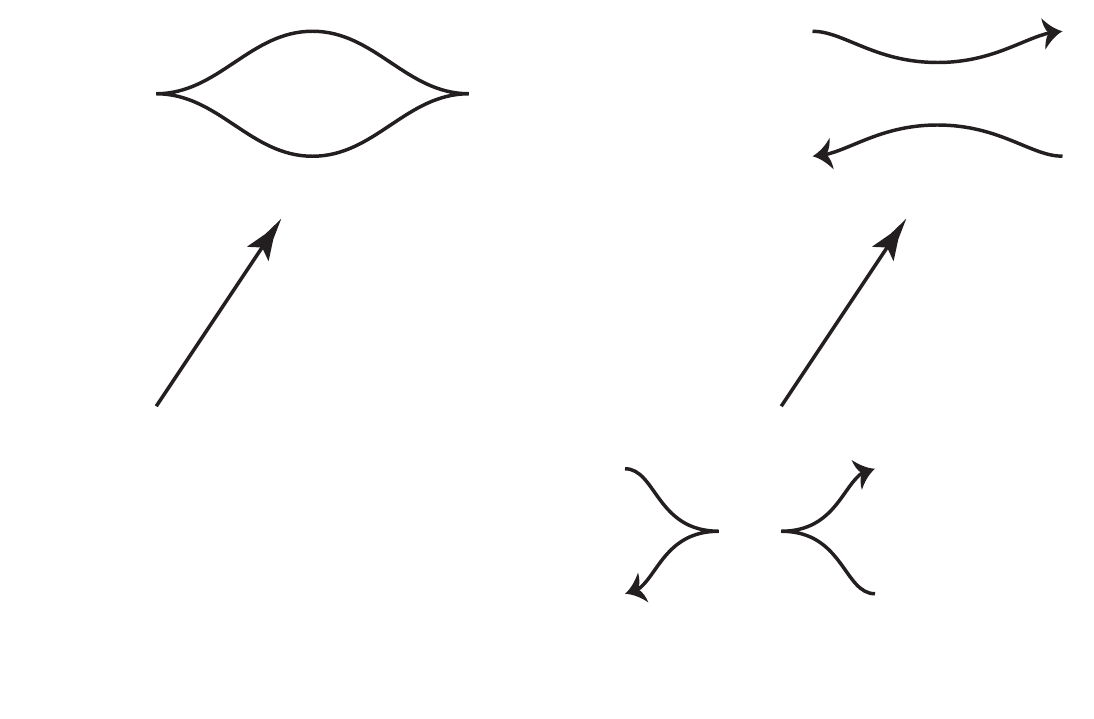}
    \caption{Diagrammatic moves corresponding to the attachment of a $0$-handle (left) and a $1$-handle (right).}
    \label{fig:construct}
\end{figure}

In particular, any decomposable cobordism can be completely described by a finite list of moves which, when applied to the front diagram $\Lambda_-$, results in the front diagram $\Lambda_+$. 

We note that the genus of a decomposable Lagrangian cobordism $L$ is given by the following: 
\begin{equation} \label{eq:decomp-genus}
g(L) = \frac{1}{2}\big((\text{\# of 1-handles}) - (\text{\# of 0-handles})\big)
\end{equation}

For a decomposable concordance, the following lemma describes how the topological condition that a concordance is diffeomorphic to a cylinder is reflected in the combinatorics of the sequence of moves used to construct the concordance.

\begin{lemma}\label{1hcomponent} 
For a decomposable concordance $L$ between Legendrian knots, any 1-handle in $L$ that takes $\Lambda$ to $\Lambda'$ must join strands of two different components in $\Lambda$ and result in one component in $\Lambda'$.
\end{lemma}

\begin{proof} 
Suppose $\Lambda_- \prec_L \Lambda_+$ is a decomposable concordance between knots. Applying Equation~(\ref{eq:decomp-genus}), we know that the total number of 1-handles is equal to the total number of 0-handles in $L$. Observe that any 0-handle increases the number of components by one. Since $\Lambda_+$ also has exactly one component and isotopy will never change the number of components, it must be the case that each 1-handle decreases the number of components by one. The only way this can happen is if each 1-handle joins two distinct components of $\Lambda$.
\end{proof}

\section{Existence of Decomposable Concordances}\label{mainProof}

We now have the language in place to prove Theorem~\ref{mainThrm}, namely that twisted Legendrian satellites between decomposably concordant knots are also decomposably concordant.

Suppose we have a decomposable concordance $\Lambda_- \prec_L \Lambda_+$. We will build the concordance $L_\Sigma$ from $\lsat{\Lambda_-}{\Pi + \tw_m}$ to $\lsat{\Lambda_+}{\Pi + \tw_m}$ inductively  by enumerating the moves used to produce $L$ and constructing corresponding sequences of moves for $L_\Sigma$. In particular, we claim that given an elementary cobordism $L'$ from the $k$-component link $\Lambda$ to the $k'$-component link $\Lambda'$, we can construct a corresponding sequence of moves on $\lsat {\Lambda} {\mathbf{\Psi}}$ that result in the link $\lsat {\Lambda'} {\mathbf{\Psi'}}$, where $\Psi_1 = \Psi'_1= \Pi+\tw_n$ and $\Psi_i = \Psi_i' = \tw_n$ for all $i > 1$. 

The proof proceeds by examining each type of elementary cobordism in turn.  If the original elementary cobordism $L'$ comes from a Legendrian isotopy, then the inductive claim follows directly from Lemma~\ref{satLiftIso}.  If $L'$ comes from attaching a $0$-handle, then, by definition, $\Lambda'$ is exactly the disjoint sum of $\Lambda$ with the unknot $\Upsilon$. Notice that this new component cannot be the first component of $\Lambda'$, and hence in the satellite, it will simply appear as a disjoint copy of $\lsat{\Upsilon}{\tw_n}$.   Thus, we need only construct a sequence of elementary cobordism moves which takes $\emptyset$ to $\lsat{\Upsilon}{\tw_n}$, which is simple to do as $\lsat{\Upsilon}{\tw_n}$ is isotopic to the disjoint union of $n$ copies of $\Upsilon$.

The interesting part of the proof treats the case of adding a $1$-handle to $\Lambda$.  By Lemma~\ref{1hcomponent}, we know that the $1$-handle must be between strands of two different components of $\Lambda$, which join to form a single component of $\Lambda'$. The heart of the proof is encapsulated by the following lemma:

\begin{lemma} \label{lem:1-handle-sat}
Let $\Lambda$ be a two-component Legendrian link and let $\Lambda'$ be the Legendrian knot that results from attaching a $1$-handle that joins the two components. Suppose that the $1$-handle attachment takes place in a neighborhood $\,V$ in the front diagram.  There exists a sequence of elementary cobordism moves that takes $\lsat{\Lambda}{(\tw_n,\tw_n)}$ to $\lsat{\Lambda'}{\tw_n}$, with all of the moves occurring in $V$.
\end{lemma}

\begin{proof}
We will prove the lemma by induction on $n$, the number of strands in our patterns. The base case ($n = 1$) is trivial, as we simply copy the single 1-handle move which acts on $\Lambda_-$. 

To prove the inductive case, suppose there exists a sequence of elementary cobordisms which takes $\lsat{\Lambda}{(\tw_{n-1}, \tw_{n-1})}$ to $\lsat{\Lambda'}{\tw_{n-1}}$ with all moves taking place in a neighborhood $V_{n-1}$. Consider the Legendrian link $\lsat{\Lambda}{(\tw_n, \tw_n)}$. We first ensure that the full twists of both components are placed in a neighborhood $V_n$ that contains $V_{n-1}$. We begin with the sequence of moves in Figure~\ref{fig:1-h-moves1}.

\begin{figure}
\begin{align*}
&\begin{gathered}
    \includegraphics[scale=0.4]{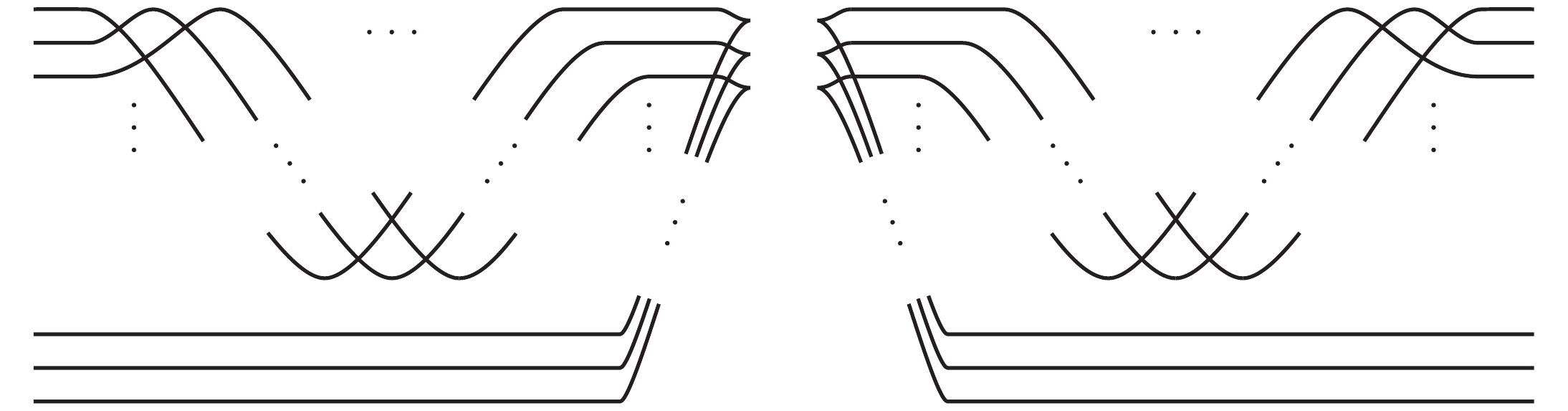}
\end{gathered} \\
\xrightarrowt{{\normalsize \enspace
    R3 $\cross \big(n(n-1)\big)$
\enspace}}{}\enspace
&\begin{gathered}
    \includegraphics[scale=0.4]{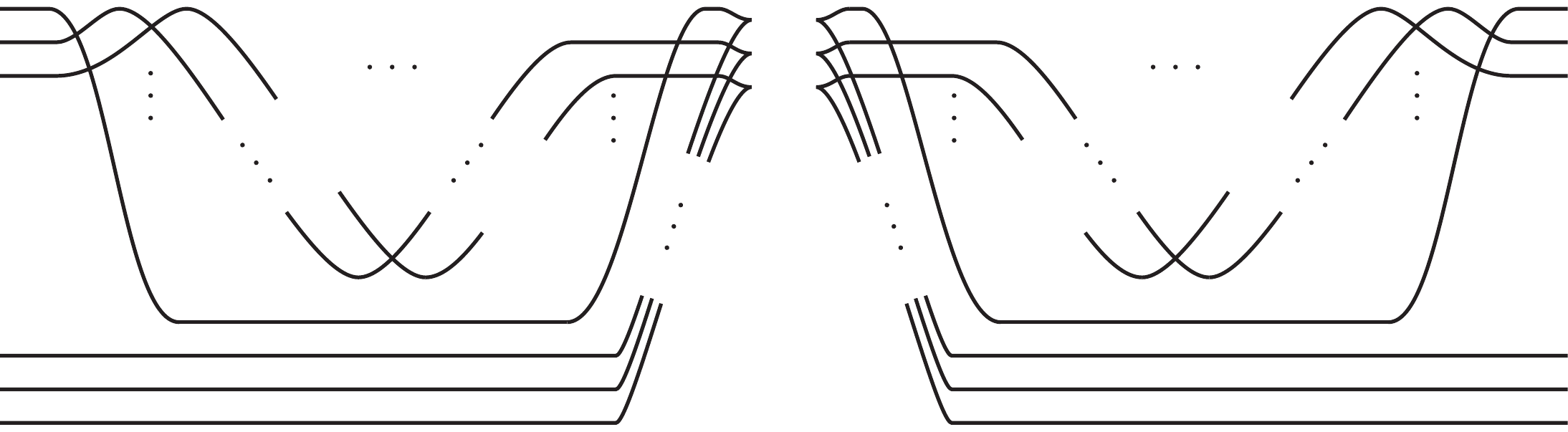}
\end{gathered} \\
\xrightarrowt{{\normalsize \enspace
    R2 $\cross \big(2n\big)$
\enspace}}{}\enspace
&\begin{gathered}
    \includegraphics[scale=0.4]{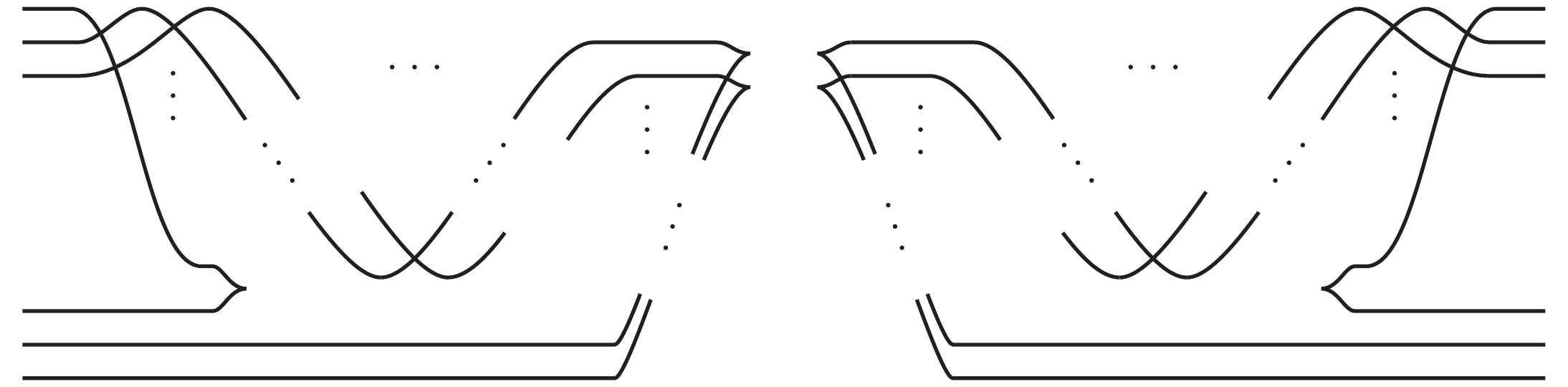}
\end{gathered}
\end{align*}
\caption{The first set of elementary cobordism moves on $\lsat{\Lambda}{(\tw_n, \tw_n)}$.}
\label{fig:1-h-moves1}
\end{figure}

Notice that at the end of this sequence of moves, in a sufficiently small neighborhood, we have exactly the diagram that occurs in a neighborhood $V_{n-1}$ for  $\lsat{\Lambda}{(\tw_{n-1},\tw_{n-1})}$. By the induction hypothesis, we can perform a series of elementary cobordism moves to take this part of the diagram to the corresponding one in $\lsat{\Lambda'}{\tw_{n-1}}$ inside $V_{n-1}$.  Finally, as shown in Figure~\ref{fig:1-h-moves2}, there is a sequence of isotopies that results in $\lsat{\Lambda'}{\tw_n}$.

\begin{figure}
\begin{align*}
\xrightarrowt{{\normalsize \enspace
    Ind. Hypothesis
\enspace}}{}\enspace
&\begin{gathered}
    \includegraphics[scale=0.45]{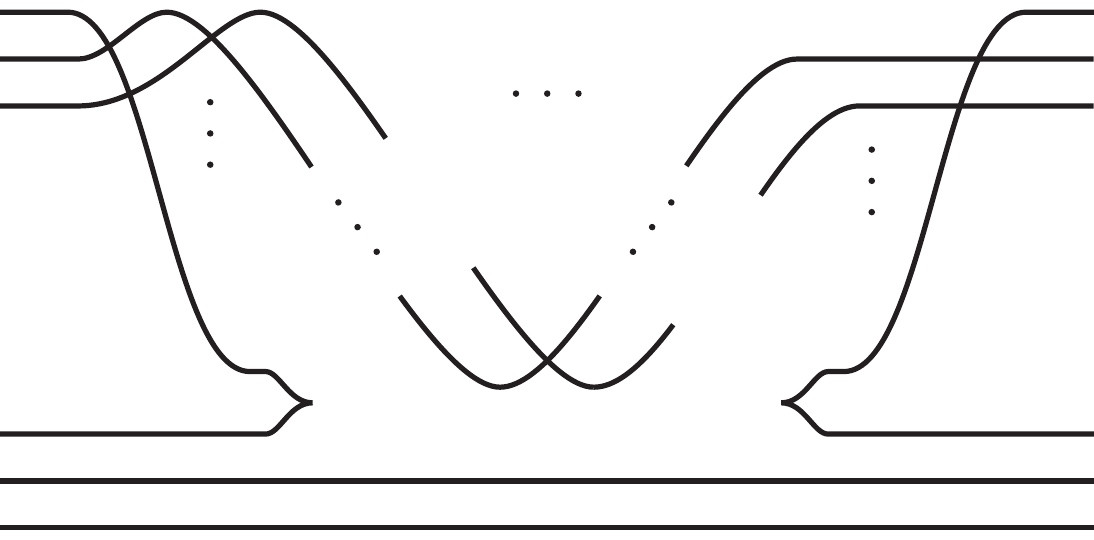}
\end{gathered} \\
\xrightarrowt{{\normalsize \enspace
    1-handle
\enspace}}{}\enspace
&\begin{gathered}
    \includegraphics[scale=0.45]{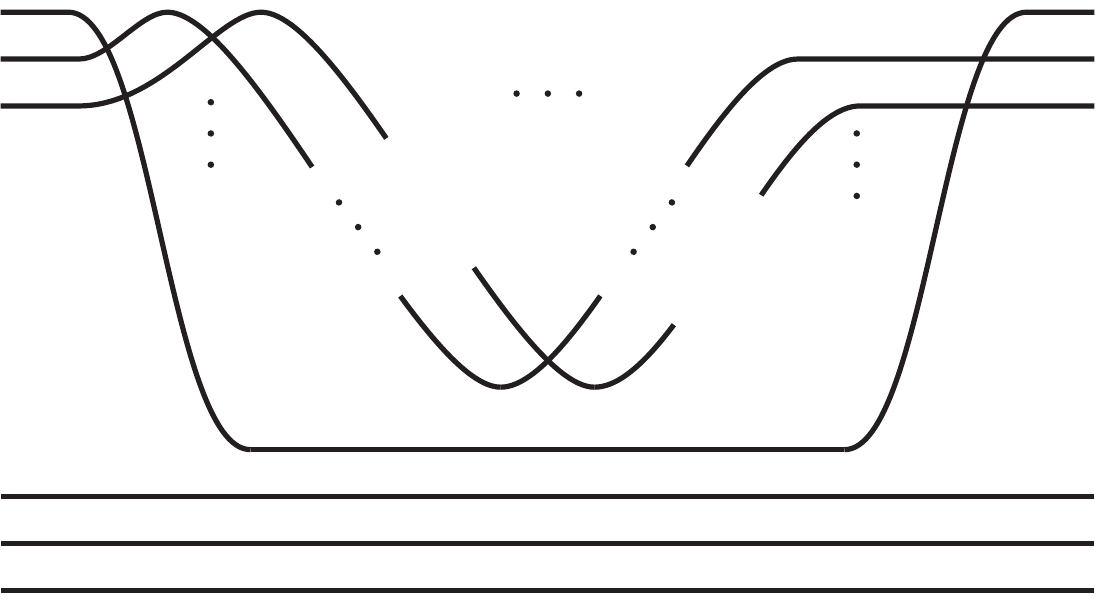}
\end{gathered} \\
\xrightarrowt{{\normalsize \enspace
    R3 $\cross \left(\frac{n(n-1)}{2}\right)$
\enspace}}{}\enspace
&\begin{gathered}
    \includegraphics[scale=0.45]{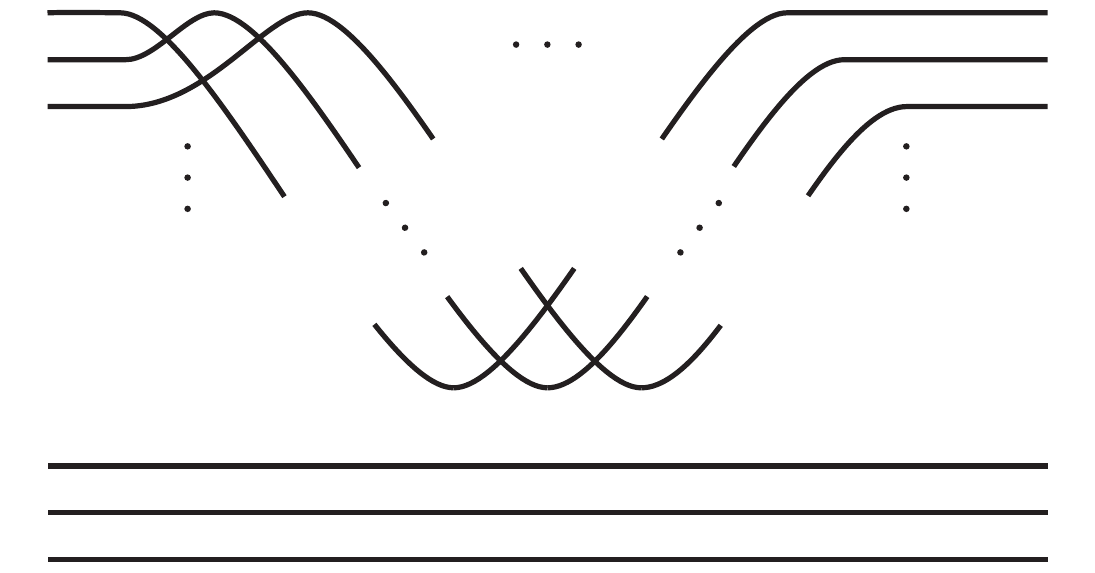}
\end{gathered}
\end{align*}
\caption{The second set of elementary cobordism moves on $\lsat{\Lambda}{(\tw_n, \tw_n)}$.}
\label{fig:1-h-moves2}
\end{figure}

Although we omitted orientations from these diagrams, notice that regardless of the pattern, since the original 1-handle on $\Lambda$ is oriented, the 1-handles in the figures must also be oriented.  Finally, note that in this construction, we used exactly $n$ 1-handles in the cobordism from $\lsat{\Lambda}{(\tw_n,\tw_n)}$ to $\lsat {\Lambda'}{\tw_n}$.
\end{proof}

To see how the lemma is used in the general construction, we may  apply it verbatim if neither component involved in the $1$-handle is the first component of $\Lambda$. If one of the components is, indeed, the first, then we may  ensure that the pattern $\Pi$ is moved far from the neighborhood $V$, and we can see that the lemma still applies. Thus, we obtain a sequence of elementary cobordism moves from $\lsat {\Lambda} {\mathbf{\Psi}}$ to $\lsat {\Lambda'} {\mathbf{\Psi'}}$. It remains to check that the resulting cobordism $L_\Sigma$ is a concordance. As discussed above, for each 1-handle (or 0-handle) in $L$, we have $n$ 1-handles (or 0-handles) in $L_\Sigma$, so indeed $\chi(L_\Sigma) = n\cdot\chi(L) = 0$.  This completes the proof of Theorem~\ref{mainThrm}.

\section{Lagrangian Slice Knots and Satellites}
\label{sec:slice}

We end with the proof of Theorem~\ref{thm:no-new-slice}, which states that if $P$ is a smooth solid torus knot for which $\sat{0}{U}{P}$ is slice and  $K$ is a slice knot with maximal Thurston-Bennequin number strictly less than $-1$, then $\sat{0}{K}{P}$ has no Lagrangian slice representatives. Suppose, for the sake of contradiction, that the smooth solid torus knot $P$ has the property that $\sat{0}{U}{P}$ is slice. Let $\Pi$ be a Legendrian representative of $P$ with strand number $m$.  Finally, to complete the setup, let $\Lambda$ be a Legendrian representative of a slice knot $K$ with $\tb(\Lambda) = \overline{\tb}(K) < -1$ and $\lsat{\Lambda}{\Pi - \tb(\Lambda) \cdot \tw_m}$ Lagrangian slice.  

Since $\lsat{\Lambda}{\Pi - \tb(\Lambda) \cdot \tw_m}$ is Lagrangian slice, we may compute as follows, using the sharpness of the slice-Bennequin inequality as in Equation (\ref{eq:filling-g4}) in the first line:
\begin{align*}
    -1 &= \tb(\lsat{\Lambda}{\Pi - \tb(\Lambda) \cdot \tw_m}) & \\
    &= \delta(\Pi)^2 \tb(\Lambda) + \tb(\Pi - \tb(\Lambda) \cdot \tw_m) && \text{by Lemma~\ref{lem:sat-tb}}\\
    &= \delta(\Pi)^2 \tb(\Lambda) + \tb(\Pi) - \tb(\Lambda) \delta(\Pi)^2 + \tb(\Lambda)m && \text{by Equation~(\ref{eq:tb-tw})} \\
    &= \tb(\Pi)+  \tb(\Lambda)m
\end{align*}

In particular, the upper bound $\tb(\Lambda) \leq -2$ tells us that
\begin{equation} \label{eq:tb-pi-1}
    \tb(\Pi) \geq 2m-1.
\end{equation}

On the other hand, since $\sat{0}{U}{P}$ is slice, the slice-Bennequin inequality (\ref{eq:slice-bennequin}) and computations similar to those above tell us that
\begin{align*}
    -1 &\geq \tb(\lsat{\Upsilon}{\Pi + \tw_m}) &\\
    &= \tb(\Pi) + m
\end{align*}

Thus, we see that 
\begin{equation} \label{eq:tb-pi-2}
    \tb(\Pi) \leq m-1.
\end{equation}

As we must have $m \geq 1$, the inequalities (\ref{eq:tb-pi-1}) and (\ref{eq:tb-pi-2}) yield the desired contradiction.

\bibliographystyle{amsplain}
\bibliography{main}

\def\cprime{$'$} \def\polhk#1{\setbox0=\hbox{#1}{\ooalign{\hidewidth
  \lower1.5ex\hbox{`}\hidewidth\crcr\unhbox0}}}
\providecommand{\bysame}{\leavevmode\hbox to3em{\hrulefill}\thinspace}
\providecommand{\MR}{\relax\ifhmode\unskip\space\fi MR }
\providecommand{\MRhref}[2]{%
  \href{http://www.ams.org/mathscinet-getitem?mr=#1}{#2}
}
\providecommand{\href}[2]{#2}
\begin{thebibliography}{10}

\bibitem{audin-lalonde-polterovich}
M.~Audin, F.~Lalonde, and L.~Polterovich, \emph{Symplectic rigidity:
  {L}agrangian submanifolds}, Holomorphic curves in symplectic geometry, Progr.
  Math., vol. 117, Birkh\"auser, Basel, 1994, pp.~271--321.

\bibitem{bs:monopole}
J.~Baldwin and S.~Sivek, \emph{Invariants of legendrian and transverse knots in
  monopole knot homology}, Preprint available on arXiv as arXiv:1405.3275,
  2014.

\bibitem{bty}
B.~Boranda, L.~Traynor, and S.~Yan, \emph{The surgery unknotting number of
  {L}egendrian links}, Involve \textbf{6} (2013), no.~3, 273--299. \MR{3101761}

\bibitem{bst:construct}
F.~Bourgeois, J.~Sabloff, and L.~Traynor, \emph{Lagrangian cobordisms via
  generating families: {C}onstruction and geography}, Algebr. Geom. Topol.
  \textbf{15} (2015), no.~4, 2439--2477.

\bibitem{polyfillability}
C.~Cao, N.~Gallup, K.~Hayden, and J.~Sabloff, \emph{Topologically distinct
  {L}agrangian and symplectic fillings}, Math. Res. Lett. \textbf{21} (2014),
  no.~1, 85--99.

\bibitem{chantraine}
B.~Chantraine, \emph{On {L}agrangian concordance of {L}egendrian knots},
  Algebr. Geom. Topol. \textbf{10} (2010), 63--85.

\bibitem{chantraine:collar}
\bysame, \emph{Some non-collarable slices of {L}agrangian surfaces}, Preprint
  available as arXiv:1108.2969v2, 2011.

\bibitem{chantraine:non-symm}
\bysame, \emph{Lagrangian concordance is not a symmetric relation}, Preprint
  available as arXiv:1301.3767, 2013.

\bibitem{chantraine:disconnected-ends}
\bysame, \emph{A note on exact {L}agrangian cobordisms with disconnected
  {L}egendrian ends}, Proc. Amer. Math. Soc. \textbf{143} (2015), no.~3,
  1325--1331.

\bibitem{c-dr-g-g}
B.~Chantraine, G.~Dimitroglou~Rizell, P.~Ghiggini, and R.~Golovko, \emph{Floer
  homology and {L}agrangian concordance}, Proceedings of the {G}\"okova
  {G}eometry-{T}opology {C}onference 2014, G\"okova Geometry/Topology
  Conference (GGT), G\"okova, 2015, pp.~76--113. \MR{3381440}

\bibitem{c-dr-g-g-cobordism}
\bysame, \emph{Floer theory for {L}agrangian cobordisms}, Preprint available as
  arXiv:1511.09471, 2015.

\bibitem{cdr:inj-sat}
T.~D. Cochran, C.~W. Davis, and A.~Ray, \emph{Injectivity of satellite
  operators in knot concordance}, J. Topol. \textbf{7} (2014), no.~4, 948--964.

\bibitem{chl:fractal-concordance}
T.~D. Cochran, S.~Harvey, and C.~Leidy, \emph{Primary decomposition and the
  fractal nature of knot concordance}, Math. Ann. \textbf{351} (2011), no.~2,
  443--508.

\bibitem{cns:obstr-concordance}
C.~Cornwell, L.~Ng, and S.~Sivek, \emph{Obstructions to {L}agrangian
  concordance}, Preprint available on arXiv as arXiv:1411.1364, 2014.

\bibitem{rizell:surgery}
G.~Dimitroglou~Rizell, \emph{Legendrian ambient surgery and {L}egendrian
  contact homology}, J. Symplectic Geom. \textbf{14} (2016), no.~3, 811--901.

\bibitem{rizell:lifting}
\bysame, \emph{Lifting pseudo-holomorphic polygons to the symplectisation of
  {$P\times\Bbb{R}$} and applications}, Quantum Topol. \textbf{7} (2016),
  no.~1, 29--105.

\bibitem{ekholm:rsft}
T.~Ekholm, \emph{Rational symplectic field theory over {$\Bbb Z_2$} for exact
  {L}agrangian cobordisms}, J. Eur. Math. Soc. (JEMS) \textbf{10} (2008),
  no.~3, 641--704.

\bibitem{ekholm:rsft-survey}
\bysame, \emph{A version of rational {SFT} for exact {L}agrangian cobordisms in
  1-jet spaces}, New perspectives and challenges in symplectic field theory,
  CRM Proc. Lecture Notes, vol.~49, Amer. Math. Soc., Providence, RI, 2009,
  pp.~173--199.

\bibitem{ekholm:lagr-cob}
\bysame, \emph{Rational {SFT}, linearized {L}egendrian contact homology, and
  {L}agrangian {F}loer cohomology}, Perspectives in analysis, geometry, and
  topology, Progr. Math., vol. 296, Birkh\"auser/Springer, New York, 2012,
  pp.~109--145.

\bibitem{ehk:leg-knot-lagr-cob}
T.~Ekholm, K.~Honda, and T.~K\'alm\'an, \emph{Legendrian knots and exact
  {L}agrangian cobordisms}, J. Eur. Math. Soc. (JEMS) \textbf{18} (2016),
  no.~11, 2627--2689.

\bibitem{etnyre:knot-intro}
J.~Etnyre, \emph{Legendrian and transversal knots}, Handbook of knot theory,
  Elsevier B. V., Amsterdam, 2005, pp.~105--185.

\bibitem{golovko:tb}
R.~Golovko, \emph{A note on {L}agrangian cobordisms between {L}egendrian
  submanifolds of $\mathbb{R}^{2n+1}$}, Pacific J. Math. \textbf{261} (2013),
  no.~1, 101--116.

\bibitem{positivity}
K.~Hayden and J.~Sabloff, \emph{Positive knots and {L}agrangian fillability},
  Proc. Amer. Math. Soc. \textbf{143} (2015), no.~4, 1813--1821.

\bibitem{hk:wh}
M.~Hedden and P.~Kirk, \emph{Instantons, concordance, and {W}hitehead
  doubling}, J. Differential Geom. \textbf{91} (2012), no.~2, 281--319.

\bibitem{kirby:list}
R.~(ed.) Kirby, \emph{Problems in low-dimensional topology}, Available at
  \texttt{www.math.berkeley.edu/$\sim$kirby}, 1995.

\bibitem{ng:satellite}
L.~Ng, \emph{The legendrian satellite construction}, Preprint available as
  arXiv:math/0112105, 2001.

\bibitem{lenny-lisa}
L.~Ng and L.~Traynor, \emph{Legendrian solid-torus links}, J. Symplectic Geom.
  \textbf{2} (2004), no.~3, 411--443.

\bibitem{rudolph:qp-obstruction}
L.~Rudolph, \emph{Quasipositivity as an obstruction to sliceness}, Bull. Amer.
  Math. Soc. (N.S.) \textbf{29} (1993), no.~1, 51--59.

\bibitem{josh-lisa:obstr}
J.~Sabloff and L.~Traynor, \emph{Obstructions to {L}agrangian cobordisms
  between {L}egendrian submanifolds}, Algebr. Geom. Topol. \textbf{13} (2013),
  2733--2797.

\end{thebibliography}

\end{document}